\documentclass[12pt]{amsart}
\usepackage{amssymb}
\usepackage{amscd}

\newtheorem{theorem}{Theorem}[section]
\newtheorem{lemma}[theorem]{Lemma}

\newtheorem{corollary}[theorem]{Corollary}

\theoremstyle{definition}
\newtheorem{definition}[theorem]{Definition}

\theoremstyle{remark}
\newtheorem{remark}[theorem]{Remark}
\numberwithin{equation}{section}

\newcommand{\B}{\mathbb{B}}
\newcommand{\C}{\mathbb{C}}
\newcommand{\E}{\mathbb{E}}

\newcommand{\K}{\mathbb{K}}

\newcommand{\N}{\mathbb{N}}

\newcommand{\Z}{\mathbb{Z}}

\newcommand\cE{\mathcal{E}}
\newcommand\cF{\mathcal{F}}
\newcommand\cH{\mathcal{H}}
\newcommand\cK{\mathcal{K}}

\newcommand\cO{\mathcal{O}}

\newcommand\Ad{\operatorname{Ad}}
\newcommand\Aut{\operatorname{Aut}}

\newcommand\id{\operatorname{id}}
\newcommand\Hom{\operatorname{Hom}}
\newcommand\End{\operatorname{End}}

\newcommand\Tr{\operatorname{Tr}}

\newcommand\inpr[2]{\langle{#1,#2}\rangle}

\newcommand\halpha{\hat{\alpha}}
\newcommand\hbeta{\hat{\beta}}
\newcommand\hotimes{\hat{\otimes}}

\newcommand\xbf{\mbox{\boldmath $\mathnormal{x}$}}

\begin{document}
\title{Quasi-free actions of finite groups on the Cuntz algebra $\cO_\infty$} 

\author{Pavle Goldstein}
\address{Department of Mathematics\\ University of Zagreb\\ 
41000 Zagreb\\ Croatia}
\email{payo@math.hr}

\author{Masaki Izumi}
\address{Department of Mathematics\\ Graduate School of Science\\
Kyoto University\\ Sakyo-ku, Kyoto 606-8502\\ Japan}
\email{izumi@math.kyoto-u.ac.jp}
\thanks{Supported in part by the Grant-in-Aid for Scientific Research (B) 22340032, JSPS}



\begin{abstract} We show that any faithful quasi-free actions of a finite group on the Cuntz algebra 
$\cO_\infty$ are mutually conjugate, and that they are asymptotically representable. 
\end{abstract}

\maketitle

\section{Introduction} 
The Cuntz algebra $\cO_n$, $n=2,3,\cdots,\infty$, is the universal $C^*$-algebra generated by isometries 
$\{s_i\}_{i=1}^n$ with mutually orthogonal ranges, satisfying $\sum_{i=1}^ns_is_i^*=1$ if $n$ is finite. 
It is well known that the two algebras $\cO_2$ and $\cO_\infty$, among the others, play special roles in the  
celebrated classification theory of Kirchberg algebras (see \cite{Ph}, \cite{R01}). 

An action $\alpha$ of a group $G$ on $\cO_n$ is said to be {\it quasi-free} if $\alpha_g(\cH_n)=\cH_n$ for all $g\in G$, 
where $\cH_n$ is the closed linear span of the generators $\{s_i\}_{i=1}^n$. 
We restrict our attention to finite $G$ throughout this note. 
To develop a $G$-equivariant version of the classification theory, it is expected that $G$-actions on $\cO_2$ with the 
Rohlin property and the quasi-free $G$-actions on $\cO_\infty$ would play similar roles as $\cO_2$ and $\cO_\infty$ do in 
the case without group actions. 
Since we have already had a good understanding of the former thanks to \cite{IDuke}, our task in this note is to investigate the latter, 
the quasi-free $G$-actions on $\cO_\infty$. 

The space $\cH_n$ has a Hilbert space structure with inner product $t^*s=\inpr{s}{t}1$, and a quasi-free $G$-action $\alpha$ 
gives a unitary representation $(\pi_\alpha,\cH_\alpha)$, where $\pi_\alpha(g)$ is the restriction of $\alpha_g$ to $\cH_\alpha$.  
It is known that the association $\alpha\mapsto \pi_\alpha$ gives a one-to-one correspondence between the quasi-free $G$ actions 
on $\cO_n$ and the unitary representations of $G$ in $\cH_n$. 
The conjugacy class of $\alpha$ depends on the unitary equivalence class of 
$(\pi_\alpha,\cH_n)$, at least a priori. 
Indeed, it really does  when $n$ is finite, and this can be seen by computing the $K$-groups of the crossed product 
(see, for example, \cite{CE}, \cite{IDuke}, \cite{IAdv}, \cite{MRS}). 
However, when $n=\infty$, the pair $(\cO_\infty,\alpha)$ is $KK_G$-equivalent to the pair $(\C,\id)$, and 
there is no way to differentiate the quasi-free actions as far as $K$-theory is concerned. 

One of the purposes of this note is to show that any two faithful quasi-free $G$-actions on $\cO_\infty$ are indeed 
mutually conjugate for every finite group $G$ (Corollary \ref{unique}). 
Our main technical result is Theorem \ref{GLP}, an equivariant version of Lin-Phillips's result  
\cite[Theorem 3.3]{LP}, and Corollary \ref{unique} follows from it via Theorem \ref{splitting}, 
an equivariant version of Kirchberg-Phillips's $\cO_\infty$ theorem \cite[Theorem 3.15]{KP}.  

Using Theorem \ref{GLP}, we also show that the quasi-free actions are asymptotically representable for any finite group $G$, 
which is another purpose of this note. 
The notion of asymptotic representability for group actions was introduced by the second-named author, 
and it is found to be important in the recent development of the classification of group actions on 
$C^*$-algebras (see \cite{IM}, \cite{MRS}).

The reader is referred to \cite{R01} for the basic properties and classification results for Kirchberg algebras. 
We denote by $\K$ the set of compact operators on a separable infinite dimensional Hilbert space. 
For a $C^*$-algebra $A$, we denote by $\tilde{A}$ and $M(A)$ the unitization and the multiplier algebra of $A$ respectively.  
When $A$ is unital, we denote by $U(A)$ the unitary group of $A$. 
For a homomorphism $\rho:A\rightarrow B$ between $C^*$-algebras $A$, $B$, we denote by $K_*(\rho)$ the 
homomorphism from $K_*(A)$ to $K_*(B)$ induced by $\rho$. 
We denote by $A\otimes B$ the minimal tensor product of $A$ and $B$. 

This work originated from the first-named author's unpublished preprint \cite{G}, where the idea of developing an equivariant 
version of Lin-Phillips's argument was introduced. 
Some results in this note are also obtained by N. C. Phillips, and the authors would like to thank him for informing of it. 
\section{Preliminaries for $G$-$C^*$-algebras}

We fix a finite group $G$. 
By a $G$-$C^*$-algebra $(A,\alpha)$, we mean a $C^*$-algebra $A$ with a fixed $G$-action $\alpha$. 
We denote by $A^G$ the fixed point algebra 
$$\{a\in A|\; \alpha_g(a)=a,\;\forall g\in G\}.$$ 

We denote by $\{\lambda_g^\alpha\}_{g\in G}$ the implementing unitary representation 
of $G$ in the crossed product $A\rtimes_\alpha G$. 
For a finite dimensional (not necessarily irreducible) unitary representation $(\pi,H_\pi)$ of $G$, 
we introduce a homomorphism 
$$\halpha_\pi:A\rtimes_\alpha G\rightarrow (A\rtimes_\alpha G)\otimes B(H_\pi),$$
which is a part of the dual coaction of $\alpha$, by 
$\halpha_\pi(a)=a\otimes 1$ for $a\in A$, and $\halpha_\pi(\lambda^\alpha_g)=\lambda^\alpha_g\otimes \pi(g)$ 
for $g\in G$. 
We denote by $\hat{G}$ the unitary dual of $G$, and by $\Z\hat{G}$ the representation ring of $G$. 
Then identifying $K_*(A\rtimes_\alpha G)$ with $K_*((A\rtimes_\alpha G)\otimes B(H_\pi))$, we get a 
$\Z \hat{G}$-module structure of $K_*(A\rtimes_\alpha G)$ from $K_*(\halpha_\pi)$. 

Let 
$$e_\alpha=\frac{1}{\#G}\sum_{g\in G}\lambda^\alpha_g,$$
which is a projection in $(A\rtimes_\alpha G)\cap {A^G}'$. 
We denote by $j_\alpha$ the homomorphism from $A^G$ into $A\rtimes_\alpha G$ defined by 
$j_\alpha(x)=xe_\alpha$. 
When $A$ is simple and $\alpha$ is outer, that is, $\alpha_g$ is outer for every $g\in G\setminus \{e\}$, 
then $K_*(j_\alpha)$ is an isomorphism from $K_*(A^G)$ onto $K_*(A\rtimes_\alpha G)$. 
When $A$ is purely infinite and simple, and $\alpha$ is outer, 
then $A^G$ and $A\rtimes_\alpha G$ are purely infinite  and simple. 

A $G$-homomorphism $\varphi$ from a $G$-$C^*$-algebra $(A,\alpha)$ into another $G$-$C^*$-algebra 
$(B,\beta)$ is a homomorphism from $A$ into $B$ intertwining the two $G$-actions $\alpha$ and $\beta$. 
Such $\varphi$ gives rise to an element in the equivariant $KK$-group $KK_G(A,B)$, 
which is denoted by $KK_G(\varphi)$. 
We denote by $\Hom_G(A,B)$ the set of nonzero $G$-homomorphisms from $(A,\alpha)$ into $(B,\beta)$. 
Two actions $\alpha$ and $\beta$ are said to be conjugate if there exists an invertible element in 
$\Hom_G(A,B)$. 
Two $G$-homomorphisms $\varphi,\psi\in \Hom_G(A,B)$ are said to be $G$-unitarily equivalent if 
there exists a unitary $u\in M(B)^G$ satisfying
$\varphi(x)=u\psi(x)u^*$ for all $x\in A$. 
They are said to be $G$-asymptotically unitarily equivalent 
if there exists a norm continuous family of unitaries $\{u(t)\}_{t\geq 0}$ in $M(B)^G$ satisfying 
$$\lim_{t\to\infty}\|\varphi(x)-\Ad u(t)\circ \psi(x)\|,\quad \forall x\in A.$$
If they satisfy the same condition with a sequence of unitaries $\{u_n\}_{n=1}^\infty$ in $M(B)^G$ 
instead of the continuous family, they are said to be $G$-approximately unitarily equivalent

For a free ultrafilter $\omega\in \beta\N\setminus \N$ and a $G$-$C^*$-algebra $(A,\alpha)$, we use the following notation: 
$$c_\omega(A)=\{(x_n)\in \ell^\infty(\N,A)|\; \lim_{n\to \omega}\|x_n\|=0\},$$
$$A^\omega=\ell^\infty(\N,A)/c_\omega(A).$$
As usual, we often omit the quotient map from $\ell^\infty(\N,A)$ onto $A^\omega$. 
We regard $A$ as a $C^*$-subalgebra of $A^\omega$ consisting of the constant sequences, and we set $A_\omega=A^\omega\cap A'$. 
We denote by $\alpha^\omega$ and $\alpha_\omega$ the $G$-actions on $A^\omega$ and $A_\omega$ induced by $\alpha$ respectively, 
and we regard $(A^\omega,\alpha^\omega)$ and $(A_\omega,\alpha_\omega)$ as $G$-$C^*$-algebras. 

\begin{lemma}\label{ultra} Let $G$ be a finite group, and let $(A,\alpha)$ be a $G$-$C^*$-algebra. 
We assume that $A$ is unital, purely infinite, and simple, and $\alpha$ is outer. 
Let $\omega\in \beta\N\setminus \N$. 
\begin{itemize}
\item[(1)] $A^\omega$ is purely infinite and simple, and $\alpha^\omega$ is outer. 
\item[(2)] If $A$ is a Kirchberg algebra, $A_\omega$ is purely infinite and simple, and $\alpha_\omega$ is outer. 
\end{itemize}
\end{lemma}

\begin{proof} (1) It is easy to show that $A^\omega$ is purely infinite and simple, 
and so it suffices to show that if $\theta\in \Aut(A)$ is outer, so is $\theta^\omega\in \Aut(A^\omega)$ induced by $\theta$. 
Assume that $\theta$ is outer and $\theta^\omega$ is inner. 
Then there exists $u=(u_n)\in U(A^\omega)$ satisfying $\Ad u=\theta^\omega$. 
We my assume that $u_n$ is a unitary for all $n\in \N$. 
Since $A$ is purely infinite, there exist a sequence of nonzero projections 
$\{p_n\}_{n=1}^\infty$ in $A$ and a sequence of complex numbers $\{c_n\}_{n=1}^\infty$ with $|c_n|=1$ 
such that $\{p_nu_np_n-c_np_n\}_{n=1}^\infty$ converges to 0. 
By replacing $u_n$ with $\overline{c_n}u_n$ if necessary. we may assume $c_n=1$. 
Since $\theta$ is outer, Kishimoto's result \cite[Lemma 1.1]{K81} shows that there exists a sequence of positive elements 
$a_n\in p_nAp_n$ with $\|a_n\|=1$ such that $\{a_n\theta(a_n)\}_{n=1}^\infty$ converging to 0. 
This is contradiction. 
Indeed, let $a=(a_n)\in A^\omega$, $p=(p_n)\in A^\omega$. 
On one hand we have $a\theta^\omega(a)=0$, and on the other hand we have the following 
$$a\theta^\omega(a)=auau^*=apupau^*=apau^*=a^2u^*\neq 0.$$
This shows that $\theta^\omega$ is outer. 

(2) The statement follows from \cite[Proposition 3.4]{KP} and \cite[Lemma 2]{N}. 
\end{proof}

Now we state two results, which are equivariant versions of well-known results in the classification theory of 
nuclear $C^*$-algebras. 
We omit their proofs, which are verbatim modifications of the original ones. 
The first one is an equivariant version of \cite[Corollary 2.3.4]{R01}.

\begin{theorem}\label{GE} Let $G$ be a finite group, and let $(A,\alpha)$ and $(B,\beta)$ be unital separable $G$-$C^*$-algebras. 
If there exist $\varphi\in \Hom_G(A,B)$ and $\psi\in \Hom_G(B,A)$ such that $\psi\circ\varphi$ is $G$-approximately 
unitarily equivalent to $\id_{(A,\alpha)}$ and $\varphi\circ \psi$ is $G$-approximately unitarily equivalent to $\id_{(B,\beta)}$, 
then the two actions $\alpha$ and $\beta$ are conjugate.  
\end{theorem}

The following result is an equivariant version of \cite[Proposition 3.13]{KP}. 

\begin{theorem}\label{GKP} Let $G$ be a finite group, and let $(A,\alpha)$, $(B,\beta)$ be unital separable $G$-$C^*$ algebras. 
We regard the minimal tensor product $B\otimes B$ as a $G$-$C^*$-algebra with the diagonal action $\alpha\otimes \alpha$, 
and define $\rho_l,\rho_r\in \Hom_G(B,B\otimes B)$ by $\rho_l(x)=x\otimes 1$ and $\rho_r(x)=1\otimes x$ for $x\in B$. 
We assume that $\rho_l$ and $\rho_r$ are $G$-approximately unitarily equivalent. 
Then if there exists a unital homomorphism in $\Hom_G(B,A_\omega)$ with $\omega\in \beta \N\setminus \N$, 
the two $G$-actions $\alpha$ on $A$ and $\alpha\otimes \beta$ on $A\otimes B$ are conjugate.   
\end{theorem}

\section{Equivariant R{\o}rdam's theorem}

The purpose of this section is to show the following theorem, which is an equivariant version of R{\o}rdam's theorem 
\cite[Theorem 3.6]{R93},\cite[Theorem 5.1.2]{R01}. 

\begin{theorem} \label{GR} Let $G$ be a finite group, let $\alpha$ be a quasi-free action of $G$ on $\cO_n$ with finite $n$, 
and let $(B,\beta)$ be a $G$-$C^*$-algebra. 
We assume that $B$ is unital, purely infinite, and simple, and $\beta$ is outer. 
For two unital $G$-homomorphisms $\varphi,\psi\in \Hom_G(\cO_n,B)$, we set 
$$u_{\psi,\varphi}=\sum_{i=1}^n\psi(s_i)\varphi(s_i)^*\in U(B^G).$$ 
We introduce an endomorphism $\Lambda_\varphi \in \End(B^G)$ by 
$$\Lambda_\varphi(x)=\sum_{i=1}^n\varphi(s_i)x\varphi(s_i^*),\quad x\in B^G.$$
Then the following conditions are equivalent. 
\begin{itemize}
\item [(1)] The $G$-homomorphisms $\varphi$ and $\psi$ are $G$-approximately unitarily equivalent.   
\item [(2)] The unitary $u_{\psi,\varphi}$ belongs to the closure of $\{v\Lambda_\varphi(v^*)\in U(B^G)|\;v\in B^G\}$. 
\item [(3)] The $K_1$-class $[u_{\psi,\varphi}]\in K_1(B^G)$ is in the image of $1-K_1(\Lambda_\varphi)$. 
\item [(4)] The $K_1$-class $K_1(j_\beta)([u_{\psi,\varphi}])\in K_1(B\rtimes_\beta G)$ is in the image of $1-K_1(\hat{\beta}_{\pi_\alpha})$.
\item [(5)] The equality $KK_G(\varphi)=KK_G(\psi)$ holds in $KK_G(\cO_n,B)$. 
\end{itemize}
\end{theorem}

\begin{proof} 
The equivalence of (1) and (2) follows from $\psi(s_i)=u_{\psi,\varphi}\varphi(s_i)$ and 
$v\varphi(s_i)v^*=v\Lambda_\varphi(v^*)\varphi(s_i)$. 

The implication from (2) to (3) is trivial. 
In view of the proof of \cite[Theorem 3.6]{R93}, the implication from (3) to (2) is reduced to the Rohlin property of 
the shift automorphism of $(\bigotimes_\Z M_n(\C))^G$, where the $G$-action of the UHF algebra $\bigotimes_\Z M_n(\C)$ 
is the product action $\bigotimes_\Z\Ad \pi_\alpha(g)$. 
This follows from Kishimoto's result \cite[Theorem 2.1]{K98} (see \cite[Lemma 5.5]{IDuke} for details). 

The equivalence of (3) and (4) follows from Lemma \ref{Kcomputation2} below. 

We will show the equivalence of (4) and (5) in Appendix as it follows from a rather lengthy computation, and we do not really require it 
in the rest of this note. 
\end{proof}

To show the equivalence of (3) and (4), we first recall the following well-known fact. 

\begin{lemma}\label{Kcomputation1} 
Let $A$ be a $C^*$-algebra, and let $\{t_i\}_{i=1}^n\subset M(A)$ be isometries with mutually 
orthogonal ranges. 
Let $\{e_{ij}\}_{i,j=1}^n$ be the system of matrix units of the matrix algebra $M_n(\C)$. 
We define two homomorphisms 
$\rho_1:A\rightarrow A\otimes M_n(\C)$ and $\rho_2:A\otimes M_n(\C)\rightarrow A$ by 
$\rho_1(a)=a\otimes e_{11}$, and $\rho_2(a\otimes e_{ij})=t_iat_j^*$. 
Then $K_*(\rho_2)$ is the inverse of $K_*(\rho_1)$.  
\end{lemma}

\begin{proof} Since $K_*(\rho_1)$ is an isomorphism, it suffices to show that the homomorphism 
$\rho_2\circ \rho_1(x)=t_1xt_1^*$ induces the identity on $K_*(A)$. 
This follows from a standard argument. 
\end{proof} 

Recall that we regard $K_*(\hbeta_{\pi_\alpha})$ as an element of $\End(K_*(B\rtimes_\beta G))$ by 
identifying $K_*(B\rtimes_\beta G)$ with $K_*((B\rtimes_\beta G)\otimes B(\cH_n))$.

\begin{lemma}\label{Kcomputation2} With the above notation, we have the equality $K_*(j_\beta)\circ K_*(\Lambda_\varphi)
=K_*(\hbeta_{\pi_\alpha})\circ K_*(j_\beta)$. 
\end{lemma} 

\begin{proof} Identifying $B(\cH_n)$ with the linear span of $\{s_is_j^*\}_{i,j=1}^n$ acting on $\cH_n$ by left multiplication, 
we have 
$$\pi_\alpha(g)=\sum_{i=1}^n \alpha_g(s_i)s_i^*.$$  
We define a homomorphism 
$\rho:(B\rtimes_\beta G)\otimes B(\cH_n)\rightarrow  B\rtimes_\beta G$ by 
$\rho(x\otimes s_is_j^*)=\varphi(s_i)x\varphi(s_j)^*$, which plays the role of $\rho_2$ in Lemma \ref{Kcomputation1} 
with $A=B\rtimes_\beta G$ and $t_i=\varphi(s_i)$. 
Then for $x\in B^G$, we have 
\begin{eqnarray*}\lefteqn{
\rho\circ \hbeta_{\pi_\alpha}\circ j_\beta(x) =\frac{1}{\# G}\sum_{g\in G}\rho\circ\hbeta_{\pi_\alpha}(\lambda^\beta_g x) 
=\frac{1}{\# G}\sum_{g\in G}\rho(\lambda^\beta_g x\otimes \pi_\alpha(g))} \\
 &=&\frac{1}{\# G}\sum_{g\in G}\sum_{i=1}^n\rho(\lambda^\beta_g x\otimes \alpha_g(s_i)s_i^*) 
 =\frac{1}{\# G}\sum_{g\in G}\sum_{i=1}^n \varphi(\alpha_g(s_i))\lambda^\beta_g x\varphi(s_i)^* \\
 &=&\frac{1}{\# G}\sum_{g\in G}\sum_{i=1}^n \lambda^\beta_g \varphi(s_i) x\varphi(s_i)^* 
 =j_\beta\circ \Lambda_\varphi(x), 
\end{eqnarray*}
which proves the statement thanks to Lemma \ref{Kcomputation1}. 
\end{proof}

\section{Equivariant Lin-Phillips's theorem}

The purpose of this section is to show the following theorem, which is an equivariant version of 
Lin-Phillips's theorem \cite[Theorem 3.3]{LP}, \cite[Proposition 7.2.5]{R01}.   

\begin{theorem}\label{GLP} Let $G$ be a finite group, let $\alpha$ be a quasi-free action of $G$ on $\cO_\infty$, 
and let $(B,\beta)$ be a unital $G$-$C^*$-algebra. 
We assume that $B$ is purely infinite and simple, and $\beta$ is outer. 
Then any two unital $G$-homomorphisms in $\Hom_G(\cO_\infty,B)$ are $G$-approximately unitarily equivalent. 
\end{theorem}

Until the end of this section, we assume that $G$, $(\cO_\infty,\alpha)$ and $(B,\beta)$ are as in Theorem \ref{GLP}. 
To prove Theorem \ref{GLP}, we basically follow Lin-Phillips's strategy based on Theorem \ref{GR} in place of 
\cite[Theorem 3.6]{R93}, though we will 
take a short cut by using a ultraproduct technique.

Let $n$ be a natural number larger than 2, and let $\cE_n$ be the Cuntz-Toeplitz algebra, 
which is the universal $C^*$-algebra generated by isometries $\{t_i\}_{i=1}^n$ with mutually orthogonal ranges. 
Note that $p_n=1-\sum_{i=1}^n t_it_i^*$ is a non-zero projection not as in the case of the Cuntz algebras. 
We denote by $\cK_n$ the linear span of $\{t_i\}_{i=1}^n$. 
Quasi-free actions on $\cE_n$ are defined as in the case of the Cuntz algebras. 
For a quasi-free action $\gamma$ of $G$ on $\cE_n$, we denote by $(\pi_\gamma,\cK_n)$ the corresponding 
unitary representation of $G$ in $\cK_n$.

\begin{lemma} Let $\gamma$ be a quasi-free action of $G$ on $\cE_n$ with finite $n$, and let 
$\varphi,\psi\in \Hom_G(\cE_n,B)$ be injective $G$-homomorphisms, either both unital or both nonunital. 
If $[\varphi(1)]=[\psi(1)]=0$ in $K_0(B^G)$, then $\varphi$ and $\psi$ are $G$-approximately unitarily equivalent. 
\end{lemma}

\begin{proof} In the same way as in the proof of Lemma \ref{Kcomputation2}, we can show 
$$K_0(j_\beta)([\varphi(p_n)])=K_0(j_\beta)([\varphi(1)])-K_0(\hbeta_{\pi_\gamma})\circ K_0(j_\beta)([\varphi(1)])=0,$$ 
in $K_0(B\rtimes_\beta G)$. 
This implies $[\varphi(p_n)]=0$ in $K_0(B^G)$, and for the same reason, $[\psi(p_n)]=0$ in $K_0(B^G)$.  
Thus the statement follows from essentially the same argument as in the proof of \cite[Proposition 1.7]{LP} by 
using Theorem \ref{GR} in place of \cite[Theorem 3.6]{R93}. 
\end{proof}

Since every quasi-free $G$-action on $\cO_\infty$ is the inductive limit of a system of quasi-free actions of the form 
$\{(\cE_{n_k},\gamma^{(k)})\}_{k=1}^\infty$, we get

\begin{corollary}\label{G0} Let $\varphi,\psi\in \Hom_G(\cO_\infty,B)$ be either both unital or both nonunital. 
If $[\varphi(1)]=[\psi(1)]=0$ in $K_0(B^G)$, then $\varphi$ and $\psi$ are $G$-approximately unitarily equivalent. 
\end{corollary}

Let $\omega\in \beta\N\setminus \N$ be a free ultrafilter, and let 
$\iota_\omega:\cO_\infty\rightarrow {\cO_\infty}^\omega$ the inclusion map. 
For $\varphi\in \Hom_G(\cO_\infty,B)$, we denote by $\varphi^\omega$ the $G$-homomorphism in 
$\Hom_G({\cO_\infty}^\omega,B^\omega)$ induced by $\varphi$. 
Then it is easy to show the following three conditions for $\varphi,\psi\in \Hom_G(\cO_\infty,B)$ are
equivalent: 
\begin{itemize}
\item[(1)] $\varphi$ and $\psi$ are $G$-approximately unitarily equivalent, 
\item[(2)] $\varphi^\omega\circ \iota_\omega$ and $\psi^\omega\circ \iota_\omega$ are 
$G$-approximately unitarily equivalent,
\item[(3)] $\varphi^\omega\circ \iota_\omega$ and $\psi^\omega\circ \iota_\omega$ are 
$G$-unitarily equivalent. 
\end{itemize}
Note that since $G$ is a finite group, we have $({\cO_\infty}_\omega)^G=(\cO_\infty^G)^\omega\cap \cO_\infty'$ and 
$(B^\omega)^G=(B^G)^\omega$. 

\begin{proof}[Proof of Theorem \ref{GLP}] 
Let $\varphi,\psi\in \Hom_G(\cO_\infty,B)$ be unital. 
Since $\cO_\infty$ is a Kirchberg algebra, the $\omega$-central sequence algebra ${\cO_\infty}_\omega$ is purely infinite and simple.  
Let $H$ be the kernel of $\alpha:G\rightarrow \Aut(\cO_\infty)$. 
Since $\alpha$ is quasi-free, we may regard $\alpha$ as an outer action of $G/H$, and so  
$\alpha_\omega$ is outer as an action of $G/H$. 
This implies that $({\cO_\infty}_\omega)^G$ is purely infinite and simple. 

Choosing three nonzero projections $q_1,q_2,q_3\in ({\cO_\infty}_\omega)^G$ satisfying $q_1+q_2+q_3=1$ and 
$[1]=[q_1]=[q_2]=-[q_3]$ in $K_0(({\cO_\infty}_\omega)^G)$, 
we introduce $\varphi_i,\psi_i\in \Hom_G(\cO_\infty,B^\omega)$, $i=1,2,3$, by $\varphi_i(x)=\varphi^\omega(q_ix)$ and 
$\psi_i(x)=\psi^\omega(q_ix)$ for $x\in \cO_\infty$. 
Then we have 
$$\varphi(x)=\varphi_1(x)+\varphi_2(x)+\varphi_3(x),\quad x\in \cO_\infty,$$ 
$$\psi(x)=\psi_1(x)+\psi_2(x)+\psi_3(x),\quad x\in \cO_\infty,$$
$$[1]=[\varphi_1(1)]=[\varphi_2(1)]=-[\varphi_3(1)]=[\psi_1(1)]=[\psi_2(1)]=-[\psi_3(1)]\in K_0((B^\omega)^G).$$

Since $[(\varphi_2+\varphi_3)(1)]=[(\psi_2+\psi_3)(1)]=0$ in $K_0((B^\omega)^G)$, Corollary \ref{G0} implies that 
there exists a unitary $u\in U((B^\omega)^G)$ satisfying $u(\varphi_2+\varphi_3)(x)u^*=(\psi_2+\psi_3)(x)$ for 
$x\in \cO_\infty$. 
We set $\varphi_1^u(x)=u\varphi_1(x)u^*$. 
Then $\varphi^u_1$ is in $\Hom_G(\cO_\infty,B^\omega)$ satisfying $\varphi_1^u(1)=\psi_1(1)$, 
and $\varphi^\omega\circ \iota_\omega$ and $\varphi_1^u+\psi_2+\psi_3$ are $G$-approximately unitarily equivalent.  
Since $(\varphi_1^u+\psi_3)(1)=(\psi_1+\psi_3)(1)$ whose class in $K_0((B^\omega)^G)$ is 0, 
Corollary \ref{G0} again implies that there exists a unitary $v\in U((B^\omega)^G)$ satisfying $v\psi_2(1)=\psi_2(1)$ and 
$v(\varphi_1^u+\psi_3)(x)v^*=(\psi_1+\psi_3)(x)$ for $x\in \cO_\infty$. 
This shows that $vu\varphi(x)u^*v^*=\psi(x)$ for $x\in \cO_\infty$, and so $\varphi$ and $\psi$ 
are $G$-approximately unitarily equivalent. 
\end{proof}

\section{Splitting theorem and Uniqueness theorem}
Thanks to Theorem \ref{GLP}, we can obtain a $G$-equivariant version of Kirchberg-Phillips's $\cO_\infty$ 
theorem \cite[Theorem 3.15]{KP}, \cite[Theorem 7.2.6]{R01}. 

\begin{theorem}\label{splitting} Let $G$ be a finite group, and let $(A,\alpha)$ be a $G$-$C^*$-algebra. 
We assume that $A$ is a unital Kirchberg algebra and $\alpha$ is outer.  
Let $\{\gamma^{(i)}\}_{i=1}^\infty$ be any sequence of quasi-free actions of $G$ on $\cO_\infty$. 
Then $(A,\alpha)$ is conjugate to 
$$(A\otimes \bigotimes_{i=1}^\infty\cO_\infty, \alpha\otimes \bigotimes_{i=1}^\infty \gamma^{(i)}).$$
\end{theorem} 

\begin{proof} Let 
$$(B,\beta)=(\bigotimes_{i=1}^\infty\cO_\infty,\bigotimes_{i=1}^\infty \gamma^{(i)}),$$
and let $\rho_l,\rho_r\in \Hom_G(B,B\otimes B)$ be as in Theorem \ref{GKP}. 
Then Theorem \ref{GLP} implies that $\rho_r$ and $\rho_l$ are $G$-approximately unitarily equivalent. 

To prove the statement applying Theorem \ref{GKP}, it suffices to construct a unital embedding of $(B,\beta)$ in 
$(A_\omega,\alpha_\omega)$. 
For this, it suffices to construct a unital embedding of $(\cO_\infty,\gamma^{(i)})$ into $(A_\omega,\alpha_\omega)$ 
for each $i$ because the usual trick of taking subsequences can make the embeddings commute with each other. 
Let $\gamma$ be the quasi-free action of $G$ on $\cO_\infty$ such that $(\pi_\gamma,\cH_\infty)$ is unitarily 
equivalent to the infinite direct sum of the regular representation. 
Since there is a unital embedding of $(\cO_\infty,\gamma^{(i)})$ into $(\cO_\infty,\gamma)$, 
in order to prove the theorem, it only remains to construct a unital embedding of $(\cO_\infty,\gamma)$ into 
$(A_\omega,\alpha_\omega)$.  

Thanks to \cite[Lemma 3]{N}, we can find a nonzero projection $e\in A_\omega$ satisfying $e\alpha_{\omega g}(e)=0$ 
for any $g\in G\setminus \{e\}$. 
We choose an isometry $v\in A_\omega$ satisfying $vv^*\leq e$, and set $s_{0,g}=\alpha_{\omega g}(v)$. 
Then $\{s_{0,g}\}_{g\in G}$ are isometries in $A_\omega$ with mutually orthogonal ranges satisfying 
$\alpha_{\omega g}(s_{0,h})=s_{0,gh}$. 
Let $p=\sum_{g\in G}s_{0,g}s_{0,g}^*$, which is a projection in $(A_\omega)^G$. 
Replacing $v$ if necessary, we may assume that $p\neq 1$. 
Since $(A_\omega)^G$ is purely infinite and simple, we can find a sequence of partial isometries $\{w_i\}_{i=0}^\infty$ 
in $(A_\omega)^G$ with $w_0=p$ such that $w_i^*w_i=p$ for all $i$, and $\{w_iw_i^*\}_{i=0}^\infty$ are mutually orthogonal. 
Let $s_{i,g}=w_is_{0,g}$. 
Then $\{s_{i,g}\}_{(i,g)\in \N\times G}$ is a countable family of isometries in $A_\omega$ with mutually orthogonal ranges satisfying 
$\alpha_{\omega g}(s_{i,h})=s_{i,gh}$. 
Thus we get the desirable embedding of $(\cO_\infty,\gamma)$ into $(A_\omega,\alpha_\omega)$. 
\end{proof}

Applying Theorem \ref{splitting} to $A=\cO_\infty$ with a faithful quasi-free action $\alpha$, we obtain 

\begin{corollary}\label{unique} Any two faithful quasi-free actions of a finite group on $\cO_\infty$ are mutually conjugate. 
\end{corollary}

\section{Asymptotic representability}

\begin{definition}
An action $\alpha$ of a discrete group $G$ on a unital $C^*$-algebra $A$ is said to be \textit{asymptotically representable} if 
there exists a continuous family of unitaries $\{u_g(t)\}_{t\geq 0}$ in $U(A)$ for each $g\in G$ satisfying 
$$\lim_{t\to \infty}\|u_g(t)xu_g(t)^*-\alpha_g(x)\|=0,\quad \forall x\in A,\;\forall g\in G, $$
$$\lim_{t\to\infty} \|u_g(t)u_h(t)-u_{gh}(t)\|=0,\quad \forall g,h\in G,$$
$$\lim_{t\to\infty}\|\alpha_g(u_h(t))-u_{ghg^{-1}}(t)\|=0,\quad \forall g,h\in G.$$
An action $\alpha$ is said to approximately representable if $\alpha$ satisfies the above condition with a sequence 
$\{u_g(n)\}_{n\in \N}$ in place of the continuous family $\{u_g(t)\}_{t\geq 0}$. 
\end{definition}

Every asymptotically representable action is approximately representable, but the converse may not be true in general. 
When $G$ is a finite abelian group, an action $\alpha$ is approximately representable if and only if its dual action has 
the Rohlin property. 
When $G$ is a cyclic group of prime power order, approximately representable quasi-free actions on $\cO_n$ with finite $n$ 
are completely characterized in \cite{IAdv}, and there exist quasi-free actions that are not approximately representable. 

The purpose of this section is to show the following theorem: 

\begin{theorem}\label{asymp} Every quasi-free action of a finite group $G$ on $\cO_\infty$ is asymptotically representable. 
\end{theorem}

It is unlikely that one could show Theorem \ref{asymp} directly from the definition of quasi-free actions. 
Our proof uses the intertwining argument between two model actions; one is obviously quasi-free, and the other is an 
infinite tensor product action, that can be shown to be asymptotically representable. 

We first introduce the notion of $K$-trivial embeddings of the group $C^*$-algebra. 
We denote by $\{\lambda_g\}_{g\in G}$ the left regular representation of a finite group $G$. 
The group $C^*$-algebra $C^*(G)$ is the linear span of $\{\lambda_g\}_{g\in G}$. 

\begin{definition} Let $G$ be a finite group, and let $A$ be a unital $C^*$-algebra. 
An unital injective homomorphism $\rho:C^*(G)\rightarrow A$ is said to be a \textit{$K$-trivial embedding} if 
$KK(\rho)=KK(C^*(G)\ni \lambda_g\mapsto 1\in A)$. 
\end{definition}

For each irreducible representation $(\pi,H_\pi)$ of $G$, we choose an orthonormal basis $\{\xi(\pi)_i\}_{i=1}^{n_\pi}$ of 
$H_\pi$, where $n_\pi=\dim \pi$. 
We set $\pi(g)_{ij}=\inpr{\pi(g)\xi(\pi)_i}{\xi(\pi)_j}$, and  
$$e(\pi)_{ij}=\frac{n_\pi}{\# G}\sum_{g\in G}\overline{\pi(g)_{ij}}\lambda_g.$$
Then $\{e(\pi)_{ij}\}_{1\leq i,j\leq n_\pi}$ is a system of matrix units, and we have 
$$\lambda_g=\sum_{\pi\in \hat{G}}\sum_{i,j=1}^{n_\pi}
\pi(g)_{ij}e(\pi)_{ij}. $$ 
Let $C^*(G)_{\pi}$ be the linear span of $\{e(\pi)_{ij}\}_{i,j=1}^{\dim\pi}$. 
Then $C^*(G)_{\pi}$ is isomorphic to the matrix algebra $M_{n_\pi}(\C)$, and $C^*(G)$ has 
the direct sum decomposition 
$$C^*(G)=\bigoplus_{\pi\in \hat{G}}C^*(G)_\pi.$$ 
Let $\chi_\pi(g)=\Tr(\pi(g))$ be the character of $\pi$. 
Then 
$$z(\pi)=\frac{n_\pi}{\#G}\sum_{g\in G}\overline{\chi_\pi(g)}\lambda_g=\sum_{i=1}^{n_\pi}e(\pi)_{ii}$$
is the unit of $C^*(G)_\pi$. 

It is easy to show the following lemma:

\begin{lemma}\label{Kte} Let $G$ be a finite group, and let $A,B$ be unital simple purely infinite $C^*$-algebras. 
\begin{itemize}
\item [(1)] A unital injective homomorphism $\rho:C^*(G)\rightarrow A$ is a $K$-trivial embedding if and only if 
$[\rho(e(\pi)_{11})]=0$ in $K_0(A)$ for any nontrivial irreducible representation $\pi$. 
When $K_0(A)$ is torsion free, it is further equivalent to the condition that $[\rho(z(\pi))]=0$ in $K_0(A)$ for 
any nontrivial irreducible representation $\pi$. 
\item [(2)] Any two $K$-trivial unital embeddings of $C^*(G)$ into $A$ are unitarily equivalent. 
\item [(3)] If $\rho:C^*(G)\rightarrow A$ and $\sigma:C^*(G)\rightarrow B$ are $K$-trivial embeddings, 
so is the tensor product embedding $C^*(G)\ni \lambda_g\mapsto \rho(\lambda_g)\otimes \sigma(\lambda_g)\in A\otimes B$. 
\end{itemize}
\end{lemma}

We now construct a $K$-trivial embedding of $C^*(G)$ into $\cO_\infty$. 
We fix a nonzero projection $p\in \cO_\infty$ with $[p]=0$ in $K_0(\cO_\infty)$, and fix unital embeddings 
$$B(\ell^2(G))\subset \cO_2\subset p\cO_\infty p.$$
We denote by $\sigma_0:C^*(G)\rightarrow p\cO_\infty p$ the resulting embedding, and set $u_g=\sigma_0(\lambda_g)+1-p$. 
Then $\sigma:C^*(G)\ni \lambda_g\mapsto u_g\in \cO_\infty$ is a $K$-trivial embedding of $C^*(G)$ into $\cO_\infty$. 

Using $\{u_g\}_{g\in G}$, we introduce a $G$-$C^*$-algebra $(A,\alpha)$ by 
$$(A,\alpha_g)=\bigotimes_{k=1}^\infty(\cO_\infty, \Ad u_g).$$ 
More precisely, we set 
$$A_n=\bigotimes_{k=1}^n\cO_\infty,\quad u^{(n)}_g=\bigotimes_{k=1}^n u_g,$$ 
and $\alpha_g^{(n)}=\Ad u^{(n)}_g$. 
Then $(A,\alpha)$ is the inductive limit of the system $\{(A_n,\alpha^{(n)})\}_{n=1}^\infty$ with the embedding 
$\iota_n:A_n\ni x\mapsto x\otimes 1\in A_{n+1}$.  
The $C^*$-algebra $A$ is isomorphic to $\cO_\infty$, and the action $\alpha$ is outer. 

\begin{lemma}\label{asymp2} Let the notation be as above. 
\begin{itemize} 
\item [(1)] The action $\alpha$ is asymptotically representable. 
\item [(2)] The embedding $\iota_\alpha:C^*(G)\ni \lambda_g\mapsto \lambda^\alpha_g\in A\rtimes_\alpha G$ gives $KK$-equivalence. 
\end{itemize}
\end{lemma}

\begin{proof} (1) It suffices to construct a homotopy $\{v_g(t)\}_{t\in [0,1]}$ of unitary representations of 
$G$ in $A_3$ satisfying $v_g(0)=u_g\otimes 1\otimes 1$, $v_g(1)=u^{(2)}_g\otimes 1$, and $\alpha^{(3)}_g(v_h(t))=v_{ghg^{-1}}(t)$. 
Since $\{u_g\otimes 1\}_{g\in G}$, $\{u^{(2)}_g\}_{g\in G}$, and $\{1\otimes u_g\}_{g\in G}$ 
give $K$-trivial embeddings of $C^*(G)$ into $A_2$, there exist 
unitaries $w_1,w_2\in U(A_2)$ satisfying $w_1(u_g\otimes 1)w_1^*=w_2(1\otimes u_g)w_2^*=u^{(2)}_g$. 
Let $w=(w_1\otimes 1)(1\otimes w_2^*)$, which is a unitary in $A_3^G=A_3\cap\{u^{(3)}_g\}_{g\in G}'$ satisfying 
$w(u_g\otimes 1\otimes 1)w^*=u^{(2)}_g\otimes 1$. 
Since $A_3^G$ is isomorphic to a finite direct sum of $C^*$-algebras Morita equivalent to $\cO_\infty$, 
there exists a homotopy $\{w(t)\}_{t\in [0,1]}$ in $U(A_3^G)$ with $w(0)=1$ and $w(1)=w$. 
Thus $v_g(t)=w(t)(u_g\otimes 1\otimes 1)w(t)^*$ gives the desired homotopy. 

(2) We identify $B_n=A_n\rtimes_{\alpha^{(n)}}G$ with the $C^*$-subalgebra of $A\rtimes_\alpha G$ generated by $A_n$ and 
$\{\lambda^\alpha_g\}_{g\in G}$, and we denote by $\iota'_n:B_n\rightarrow B_{n+1}$ the embedding map. 
Then $A\rtimes_\alpha G$ is the inductive limit of the system $\{B_n\}_{n=1}^\infty$. 
Let $\iota_\alpha^{(n)}:C^*(G)\ni \lambda_g \mapsto \lambda^\alpha_g\in B_n$. 
Since we have $\iota'_{n}\circ \iota^{(n)}_\alpha=\iota^{(n+1)}_\alpha$, in order to prove the statement it suffices to show that 
$\iota^{(n)}_\alpha$ induces isomorphisms of the $K$-groups for every $n$. 

Since $\alpha^{(n)}$ is inner, there exists an isomorphism $\theta_n:B_n\rightarrow A_n\otimes C^*(G)$ given by 
$\theta_n(a)=a\otimes 1$ for $a\in A_n$ and $\theta_n(\lambda^\alpha_g)={u^{(n)}_g}\otimes \lambda_g$. 
Thus all we have to show is that the map $\theta_n\circ \iota^{(n)}_\alpha:C^*(G)\ni \lambda_g\mapsto 
u^{(n)}_g\otimes \lambda_g\in A_n\otimes C^*(G)$ induces isomorphisms of the $K$-groups. 
This follows from that fact that $A_n$ is isomorphic to $\cO_\infty$ and $\{u^{(n)}_g\}_{g\in G}$ gives an 
$K$-trivial embedding of $C^*(G)$ into $A_n$. 
\end{proof}

\begin{lemma}\label{unique2} For the $G$-$C^*$-algebra $(A,\alpha)$ as constructed above, any unital $\varphi\in \Hom_G(A,A)$ is 
$G$-asymptotically unitarily equivalent to $\id$. 
\end{lemma} 

\begin{proof} 
Let $B=A\rtimes_\alpha G$, and let $\halpha:B\rightarrow B\otimes C^*(G)$ be the dual coaction of $\alpha$. 
Then $\varphi$ extends to a unital endomorphism $\tilde{\varphi}$ in $\End(B)$ with $\tilde{\varphi}(\lambda^\alpha_g)=\lambda^\alpha_g$, 
which satisfies $\halpha\circ \tilde{\varphi}=(\tilde{\varphi}\otimes \id_{C^*(G)})\circ \halpha$. 
By Lemma \ref{asymp2},(2), we have $KK(\tilde{\varphi})=KK(\id_B)$. 
Thus Lemma \ref{asymp2},(1) and \cite[Theorem 4.8]{IM} imply that 
there exists a continuous family of unitaries $\{u(t)\}_{t\geq 0}$ in $A$ satisfying
$$\lim_{t\to \infty}\|u(t)xu(t)^*-\tilde{\varphi}(x)\|=0,\quad \forall x\in B.$$
Setting $x=\lambda^\alpha_g$, we know that $\{\alpha_g(u(t))-u(t)\}_{t\geq 0}$ converges to 1. 
Since $G$ is a finite group, there exists a conditional expectation from $A$ onto $A^G$, and we can construct 
a continuous family of unitaries $\{\tilde{u}(t)\}_{t\geq 0}$ in $A^G$ 
such that $\{u(t)-\tilde{u}(t)\}_{t\geq 0}$ converges to 0 by a standard perturbation argument. 
Therefore $\varphi$ and $\id$ are $G$-asymptotically unitarily equivalent.  
\end{proof}

\begin{proof}[Proof of Theorem \ref{asymp}] Let $\gamma$ be a faithful quasi-free $G$-action on $\cO_\infty$. 
Thanks to Corollary \ref{unique}, we may assume that $\cO_\infty$ has the canonical generators $\{s_i\}_{i\in J}$ 
with $G\subset J$ satisfying $\gamma_g(s_h)=s_{gh}$. 
Since $\alpha$ is asymptotically representable, it suffices to show that $\alpha$ and $\gamma$ are conjugate. 
Thanks to Theorem \ref{splitting}, the action $\alpha$ is conjugate to $\alpha\otimes \gamma$, and so there exists a 
unital embedding of $(\cO_\infty,\gamma)$ into $(A,\alpha)$. 
Thus if there exists a unital embedding of $(A,\alpha)$ into $(\cO_\infty,\gamma)$,     
Theorem \ref{GE}, Theorem \ref{GLP}, and Lemma \ref{unique2} imply that $\alpha$ and $\gamma$ are conjugate. 
Since $\gamma$ is conjugate to the infinite tensor product of its copies thanks to Theorem \ref{splitting} again, 
all we have to show is that there exists a unital embedding of $(\cO_\infty,\Ad u_\cdot)$ into $(\cO_\infty,\gamma)$.  

We denote by $\cO_\infty^\gamma$ the fixed point subalgebra of $\cO_\infty$ under 
the $G$-action $\gamma$. 
Since $\cO_\infty^\gamma$ is purely infinite and simple, we can choose a nonzero projection $q_0\in \cO_\infty^G$ with $[q_0]=0$ 
in $K_0(\cO_\infty^\gamma)$. 
We set $q_1=\sum_{g\in G}s_gq_0s_g^*$. 
A similar argument as in the proof of Lemma \ref{Kcomputation2} implies that $[q_1]=0$ in $K_0(\cO_\infty^\gamma)$. 
We set 
$$v_g=\sum_{h\in G}s_{gh}q_0s_h^*+1-q_1.$$
Then $\{v_g\}_{g\in G}$ is a unitary representation of $G$ in $\cO_\infty$ satisfying $\gamma_g(v_h)=v_{ghg^{-1}}$, and so 
$\{v_g^*\}_{g\in G}$ is a $\gamma$-cocycle. 
We show that this is a coboundary by using \cite[Remark 2.6]{IDuke}. 
Indeed, we have 
\begin{eqnarray}\label{E1}
\lefteqn{\frac{1}{\# G}\sum_{g\in G}v_g^*\lambda^\gamma_g=(1-q_1)e_\gamma
+\frac{1}{\# G}\sum_{g\in G}\sum_{h\in G}s_hq_0s_{gh}^*\lambda^\gamma_g} \\
 &=&(1-q_1)e_\gamma+\frac{1}{\# G}\sum_{g\in G}\sum_{h\in G}s_hq_0\lambda^\gamma_gs_{h}^*
 =(1-q_1)e_\gamma+\sum_{h\in G}s_hq_0e_\gamma s_{h}^*.\nonumber
\end{eqnarray}
This means that the class of this projection in $K_0(\cO_\infty \rtimes_\gamma G)$ is 
$$[(1-q_1)e_\gamma]+\# G[q_0 e_\gamma]=[e_\gamma],$$  
which implies that $\{v_g^*\}_{g\in G}$ is a coboundary. 
Thus there exists a unitary $v\in \cO_\infty$ satisfying $v_g^*=v\gamma_g(v^*)$. 

We set $w_g=v^*v_gv$, and claim that $\{w_g\}_{g\in G}$ gives a $K$-trivial embedding of $C^*(G)$ into $\cO_\infty^\gamma$. 
Indeed, 
$$\gamma_g(w_h)=\gamma_g(v^*)\gamma_h(v_h)\gamma_g(v)=v^*v_g^*v_{ghg^{-1}}v_gv=w_h,$$
which shows $w_g\in \cO_\infty^\gamma$. 
Let $\rho:C^*(G)\ni\lambda_g\mapsto w_g\in \cO_\infty^\gamma$. 
Thanks to Lemma \ref{Kte},(1), in order to prove the claim 
it suffices to show that $[\rho(e(\pi)_{11})]=0$ in $K_0(\cO_\infty^\gamma)$ for any nontrivial 
irreducible representation $(\pi,H_\pi)$ of $G$. 
Indeed, we have 
\begin{eqnarray*}
\lefteqn{K_0(j_\gamma)([\rho(e(\pi)_{11})])
=[\frac{n_\pi}{{\# G}^2}\sum_{g,h\in G}\overline{\pi(h)_{11}}\lambda^\gamma_gw_h]
=[\frac{n_\pi}{{\# G}^2}\sum_{g,h\in G}\overline{\pi(h)_{11}}\lambda^\gamma_gv^*v_hv]}\\
 &=&[\frac{n_\pi}{{\# G}^2}\sum_{g,h\in G}\overline{\pi(h)_{11}}\gamma_g(v^*)v_{ghg^{-1}}\lambda^\gamma_gv]
 =[\frac{n_\pi}{{\# G}^2}\sum_{g,h\in G}\overline{\pi(h)_{11}}v^*v_g^*v_{ghg^{-1}}\lambda^\gamma_gv]\\
 &=&[\frac{n_\pi}{{\# G}^2}\sum_{g,h\in G}\overline{\pi(h)_{11}}v_{hg^{-1}}\lambda^\gamma_g]
 =[\frac{n_\pi}{{\# G}^2}\sum_{g,h\in G}\overline{\pi(h)_{11}}v_{h}v_g^*\lambda^\gamma_g].
 \end{eqnarray*}
Let $\rho_0:C^*(G)\ni\lambda_g\mapsto v_g\in \cO_\infty$.  
Equation (\ref{E1}) implies that this is equal to 
$$[\rho_0(e(\pi)_{11})\sum_{k\in G}s_kq_0e_\gamma s_k^*]=n_\pi[q_0e_\gamma]=0.$$
Thus the claim is shown. 

We choose a unital embedding $\mu_0: \cO_\infty \rightarrow \cO_\infty^\gamma$. 
Since both $\{\mu_0(u_g)\}_{g\in G}$ and $\{w_g\}_{g\in G}$ give $K$-trivial embeddings of $C^*(G)$ into 
$\cO_\infty^\gamma$, Lemma \ref{Kte},(2) shows that we may assume $\mu_0(u_g)=w_g$ by replacing $\mu_0$ if necessary.  
Let $\mu(x)=v\mu_0(x)v^*$. 
Then 
\begin{eqnarray*}
\gamma_g\circ \mu(x)&=&\gamma_g(v)\mu_0(x)\gamma_g(v^*)=v_gv\mu_0(x)v^*v_g^*=vw_g\mu_0(x)w_g^*v^*\\
&=&v\mu_0(u_gxu_g^*)v^*=\mu\circ \Ad u_g(x).
\end{eqnarray*}
Thus $\mu$ is the desired embedding of $(\cO_\infty,\Ad u_\cdot)$ into $(\cO_\infty,\gamma)$. 
\end{proof}

From Theorem \ref{asymp} and Lemma \ref{unique2}, we get 

\begin{corollary} Let $G$ be a finite group, and let $\gamma$ be a quasi-free action of $G$ on $\cO_\infty$. 
Then any unital $\varphi\in \Hom_G(\cO_\infty,\cO_\infty)$ is 
$G$-asymptotically unitarily equivalent to $\id$. 
\end{corollary} 
\section{Equivariant R{\o}rdam group} 

Let $A$ and $B$ be simple $C^*$-algebras. 
For simplicity we assume that $A$ and $B$ are unital.  
Following R{\o}rdam \cite[p.40]{R01}, we denote by $H(A,B)$ the set of the approximately unitary 
equivalence classes of nonzero homomorphisms from $A$ into $B\otimes \K$. 
Choosing two isometries $s_1$ and $s_2$ satisfying the $\cO_2$ relation in $M(B\otimes \K)$, 
we can define the direct sum $[\varphi]\oplus [\psi]$ of two classes $[\varphi]$ and $[\psi]$ in $H(A,B)$ to be the class of 
the homomorphism
$$A\ni x\mapsto s_1\varphi(x)s_1^*+s_2\psi(x)s_2^*\in B\otimes \K.$$
This makes $H(A,B)$ a semigroup. 
When $A$ is a separable simple nuclear $C^*$-algebra and $B$ is a Kirchberg algebra, 
the R{\o}rdam semigroup $H(A,B)$ is in fact a group. 
Moreover, if $A$ satisfies the universal coefficient theorem, it is isomorphic to $KL(A,B)$, a certain quotient of $KK(A,B)$. 

Let $G$ be a finite group, and let $\alpha$ and $\beta$ be outer $G$-actions on $A$ and $B$ respectively. 
We equip $B\otimes \K$ with a $G$-$C^*$-algebra structure by the diagonal action $\beta^s_g=\beta_g\otimes \Ad u_g$, 
where $\{u_g\}$ is a countable infinite direct sum of the regular representation of $G$.  
Then we can introduce an equivariant version $H_G(A,B)$ as the set of the $G$-approximately equivalence classes of nonzero 
$G$-homomorphisms in $\Hom_G(A,B\otimes \K)$. 

\begin{theorem}\label{GH} Let $(A,\alpha)$ and $(B,\beta)$ be unital $G$-$C^*$-algebras with outer actions $\alpha$ and $\beta$. 
We assume that $A$ is separable, simple, and nuclear, and $B$ is a Kirchberg algebra. 
Then $H_G(A,B)$ is a group.  
\end{theorem}

Let $(A,\alpha)$ and $(B,\beta)$ be as above. 
We say that $\varphi\in \Hom_G(A,B)$ is $\cO_2$-\textit{absorbing} if there exists $\varphi'\in \Hom_G(A\otimes \cO_2,B)$ 
with $\varphi=\varphi'\circ \iota_A$, where $A\otimes\cO_2$ is equipped with the $G$-action $\alpha\otimes \id_{\cO_2}$, 
and $\iota_A:A\ni x\mapsto x\otimes 1\in A\otimes \cO_2$ is the inclusion map. 
We say that $\varphi\in \Hom_G(A,B)$ is $\cO_\infty$-\textit{absorbing} if there exists a unital embedding of $\cO_\infty$ 
in $(\varphi(1)B^G\varphi(1))\cap \varphi(A)'$. 

The proof of Theorem \ref{GH} follows from essentially the same argument as in \cite[Lemma 8.2.5]{R01} with the following lemma. 

\begin{lemma} Let the notation be as above. 
\begin{itemize}
\item[(1)] Let $\varphi,\psi\in \Hom_G(A,B)$ be $\cO_2$-absorbing $G$-homomorphisms, either both unital or both nonunital. 
Then $\varphi$ and $\psi$ are $G$-approximately unitarily equivalent. 
\item[(2)] Any element in $\Hom_G(A,B)$ is $G$-approximately unitarily equivalent to a 
$\cO_\infty$-absorbing one in $\Hom_G(A,B)$. 
\end{itemize}
\end{lemma}

\begin{proof} (1) When $\varphi$ and $\psi$ are nonunital, the two projections $\varphi(1)$ and $\psi(1)$ are equivalent in 
$B^G$, and we may assume $\varphi(1)=\psi(1)$. 
Replacing $B$ with $\varphi(1)B\varphi(1)$, we may assume that $\varphi$ and $\psi$ are unital. 

Let $\gamma$ be a faithful quasi-free action of $G$ on $\cO_\infty$. 
Since $(A\otimes \cO_2,\alpha\otimes \id_{\cO_2})$ is conjugate to $(\cO_\infty\otimes \cO_2,\gamma\otimes \id_{\cO_2})$ thanks to 
\cite[Corollary 4.3]{IDuke}, it suffices to show that any unital $\varphi,\psi\in \Hom_G(\cO_\infty\otimes \cO_2,B)$ 
are $G$-approximately unitarily equivalent. 
Theorem \ref{GLP} implies that there exists $u\in U((B^\omega)^G)$ satisfying $u\varphi(x\otimes 1)u^*=\psi(x\otimes 1)$ for any 
$x\in \cO_\infty$, where $\omega\in \beta \N\setminus \N$ is a free ultrafilter. 
Let $D=(B^\omega)^G\cap \psi(\cO_\infty\otimes 1)'$. 
Then it suffices to show that the two unital homomorphisms $\rho,\sigma\in \Hom(\cO_2,D)$ 
defined by $\rho(y)=u\varphi(1\otimes y)u^*$, $\sigma(y)=\psi(1\otimes y)$ for $y\in \cO_2$, are approximately unitarily equivalent. 
Indeed, since $(B^\omega)^G\cap B'=(B_\omega)^G$ is purely infinite and simple, for any separable $C^*$-subalgebra $C$ of 
$D$ there exists a unital embedding of $\cO_\infty$ in $D\cap C'$. 
Thus essentially the same proof of \cite[Lemma 2.1.7]{Ph} shows that $\mathrm{cel}(D)$ is finite (see \cite[Lemma 2.1.1]{Ph} 
for the definition). 
Therefore $\rho$ and $\sigma$ are approximately unitarily equivalent thanks to \cite[Theorem 3.6]{R93}. 

(2) Since $(B,\beta)$ is conjugate to $(B\otimes \cO_\infty,\beta\otimes \id_{\cO_\infty})$ thanks to \cite[Corollary 2.10]{IDuke}, 
the statement follows from the same argument as in the proof of \cite[Lemma 8.2.5,(i)]{R01}. 
\end{proof}

\begin{remark}
There are two natural homomorphisms 
$$\mu:H_G(A,B)\rightarrow H(A,B),$$
$$\nu:H_G(A,B)\rightarrow H(A\rtimes_\alpha G, B\rtimes_\beta G).$$
The first one is the forgetful functor. 
Every $\varphi\in \Hom_G(A,B)$ extends to $\tilde{\varphi}\in \Hom(A\rtimes_\alpha G,B\rtimes_\beta G)$ by 
$\tilde{\varphi}(\lambda^\alpha_g)=\lambda^\beta_g$, and the second one is given by associating $[\tilde{\varphi}]\in 
H(A\rtimes_\alpha, G,B\rtimes_\beta G)$ with $[\varphi]\in H_G(A,B)$. 
The following hold for the two maps (see  \cite[Section 4]{IM} for more general treatment): 
\begin{itemize}
\item[(1)] If $\beta$ has the Rohlin property, then $\mu$ is injective, and the image of $\mu$ is  
$$\{[\rho]\in H(A,B)|\; [\beta^s_g\circ \rho]=[\rho\circ \alpha_g],\;\forall g\in G\}. $$
\item[(2)] If $\beta$ is approximately representable, then $\nu$ is injective, and the image of $\nu$ is 
$$\{[\rho]\in H(A\rtimes_\alpha G,B\rtimes_\beta G)|\; [\hbeta^s\circ \rho]=[(\rho\otimes \id_{C^*(G)})\circ \halpha]\}.$$ 
\end{itemize}
\end{remark}

\begin{remark} 
Let $\hat{H}_G(A,B)$ be the set of the $G$-asymptotically equivalence classes of nonzero 
$G$-homomorphisms in $\Hom_G(A,B\otimes \K)$. 
It is tempting to conjecture that the natural map from $\hat{H}_G(A,B)$ to the equivariant $KK$-group $KK_G(A,B)$ 
is an isomorphism, as it is the case for trivial $G$ (see \cite{Ph}).
\end{remark}
\section{Appendix} 
In this appendix, we show the equivalence of (4) and (5) in Theorem \ref{GR}. 
Since our argument works for a compact group $G$, we assume that $G$ is compact 
in what follows. 
Our proof is new even for trivial $G$. 
Let $\alpha$ be a quasi-free action of $G$ on $\cO_n$ with finite $n$, 
and let $(B,\beta)$ be a unital $G$-$C^*$-algebra. 
Now the definition of the projection $e_\beta\in B\rtimes_\beta G$ should be modified to $e_\beta=\int_G\lambda^\beta_gdg$, 
where $dg$ is the normalized Haar measure of $G$. 
For two unital $\varphi,\psi\in \Hom_G(\cO_n,B)$, we define $u_{\varphi,\psi}\in U(B^G)$ as in Theorem \ref{GR}. 

Let $\cE_n$ be the Cuntz-Toeplitz algebra with the canonical generators $\{t_i\}_{i=1}^n$.  
We denote by $q_n$ the surjection $q_n:\cE_n\rightarrow \cO_n$ sending $t_i$ to $s_i$ for $i=1,2,\cdots, n$. 
Then the kernel $J_n$ of $q_n$ is the ideal generated by $p_n=1-\sum_{i=1}^nt_it_i^*$, and is isomorphic to the compact operators $\K$. 
We denote by $i_n:J_n\rightarrow \cE_n$ the inclusion map. 
Since $\cO_n$ is nuclear, the exact sequence 

\begin{equation}\label{short} \CD
0@>>>J_n @>i_n>>\cE_n@>q_n>> \cO_n@>>> 0,
\endCD
\end{equation}
is semisplit, that is, there exists a unital completely positive lifting $l_n:\cO_n\rightarrow \cE_n$ 
of $q_n$. 
We denote by $\tilde{\alpha}$ the quasi-free action of $G$ on $\cE_n$ that is a lift of $\alpha$. 
By replacing $l_n$ with $l_n^G$ given by 
$$l_n^G(x)=\int_G \tilde{\alpha}_g\circ l_n\circ \alpha_{g^{-1}}(x)dg,\quad x\in \cO_n, $$
we see that (8.1) is a semisplit exact sequence of $G$-$C^*$-algebras. 
Thus it induces the following 6-term exact sequence of $KK_G$-groups: 
$$\CD
KK_G^0(J_n,B)@<i_n^*<<KK_G^0(\cE_n,B)@<q_n^*<<KK_G^0(\cO_n,B)\\
@V\delta VV@.@AA\delta A\\
KK_G^1(\cO_n,B)@>>q_n^*>KK_G^1(\cE_n,B)@>>i_n^*>KK_G^1(J_n,B)
\endCD .
$$

Let $H_n$ be the $n$-dimensional Hilbert space $\C^n$ with the canonical orthonormal basis $\{e_i\}_{i=1}^n$.  
We regard $H_n$ as a $\C-\C$ bimodule with a $G$-action given by $\pi_\alpha$. 
We denote by $\cF_n$ the full Fock space 
$$\cF_n=\bigoplus_{m=0}^\infty H_n^{\otimes m},$$ 
with a unitary representation $\pi_{\cF_n}$ of $G$ coming from $\pi_\alpha$. 
Identifying $t_i$ with the creation operator of $e_i$ acting on $\cF_n$, we regard $\cE_n$ as a $C^*$-subalgebra of $\B(\cF_n)$. 
With this identification, we have $J_n=\K(\cF_n)$, and  $p_n$ is the projection onto $H_n^{\otimes 0}$. 
We regard $\cF_n$ as $J_n-\C$ bimodule, which gives the $KK_G$-equivalence of $J_n$ and $\C$. 
Pimsner's computation \cite[Theorem 4.9]{P} yields the following 6-term exact sequence:  
$$\CD
KK_G^0(\C,B)@<1-[H_n]\hotimes<<KK_G^0(\C,B)@<<<KK_G^0(\cO_n,B)\\
@V\delta' VV@.@AA\delta' A\\
KK_G^1(\cO_n,B)@>>>KK_G^1(\C,B)@>>1-[H_n]\hotimes>KK_G^1(\C,B)
\endCD ,$$
where $[H_n]\hotimes$ denote the left multiplication of the class $[H_n]\in KK_G(\C,\C)$. 
Note that the identification of $KK_G^*(J_n,B)$ and $KK_G^*(\C,B)$ is given by $[\cF_n]\in KK_G(J_n,\C)$, 
and so $\delta'=\delta\circ (\cF_n\hotimes)$. 

With the Green-Julg isomorphism $h_*:KK_G^*(\C,B)\rightarrow K_*(B\rtimes_\beta G)$ (\cite[Theorem 11.7.1]{B}), 
we have the commutative diagram  
$$\CD
KK_G^*(\C,B)@>[\cH_n]\hotimes >>KK_G^*(\C,B)\\
@V h_* VV @VV h_* V\\
K_*(B\rtimes_\beta G)@>K_*(\hat{\beta}_{\pi_\alpha}) >>K_*(B\rtimes_\beta G)\\
\endCD,
$$
and so we get the following 6-term exact sequence 
$$\CD
K_0(B\rtimes_\beta G)@<1-K_0(\hat{\beta}_{\pi_\alpha})<<K_0(B\rtimes_\beta G)@<<<KK_G^0(\cO_n,B)\\
@V\delta''VV@.@AA\delta'' A\\
KK_G^1(\cO_n,B)@>>>K_1(B\rtimes_\beta G)@>>1-K_1(\hat{\beta}_{\pi_\alpha})>K_1(B\rtimes_\beta G)
\endCD ,$$
with $\delta''=\delta\circ ([\cF_n]\hotimes)\circ h_*^{-1}$. 
Now the proof of the equivalence of (4) and (5) in Theorem \ref{GR} follows from the next theorem. 

\begin{theorem}\label{boundary} With the above notation, we have 
$$\delta''(K_1(j_\beta)([u_{\psi,\varphi}]))=KK_G(\psi)-KK_G(\varphi).$$
\end{theorem}

The proof of Theorem \ref{boundary} follows from a standard and rather tedious computation below.  
In what follows, we freely use the notation in Blackadar's book \cite{B} for $KK$-theory. 
We regard $\C_1$, $C=C_0[0,1)$, and $S=C_0(0,1)$  as $G$-$C^*$-algebras with trivial $G$-actions. 

\cite[Theorem 19.5.7]{B} shows that $\delta'$ is given by the left multiplication of the class $\delta_{q_n}$ of 
the extension (\ref{short}) in $KK_G^1(\cO_n,\C)=KK_G(\cO_n,\C_1)$, whose Kasparov module 
$(E_1,\phi_1,F_1)\in \E_G(\cO_n,\C_1)$ is given as follows. 
By the Stinespring dilation of the $G$-equivariant lifting $l_n^G:\cO_n\rightarrow \cE_n\subset \B(\cF_n)$, 
we get a Hilbert space $H$ including $\cF_n$, with a unitary representation $\pi_H$ of $G$ extending $\pi_{\cF_n}$, 
satisfying the following condition:  
there is a unital $G$-homomorphism $\Phi:\cO_n\rightarrow \B(H)$ such that if $P$ is the projection from $H$ onto 
$\cF_n$, then $l_n^G(x)=P\Phi(x)P$ for any $x\in \cO_n$. 
Now we have 
$$(E_1,\phi_1,F_1)=(H\hotimes\C_1, \Phi\hotimes 1, (2P-1)\hotimes \varepsilon),$$ 
where $\varepsilon=1\oplus -1$ is the generator of $\C_1\cong C^*(\Z_2)$. 

Let $z(t)=e^{2\pi i t}$, and let $\theta$ be the element in $\Hom_G(C_0(0,1),B)$ determined by $\theta(z-1)=u_{\psi,\varphi}-1$. 
Then $h_1^{-1}\circ K_1(j_\beta)([u_{\psi,\varphi}])$ is given by 
$$KK_G(\theta)\in KK_G(C_0(0,1),B)\cong KK_G(\C_1,B).$$  

In order to compute the Kasparov product of $\delta_{q_n}\in KK_G(\cO_n,\C_1)$ and $KK_G(\theta)\in KK_G(S,B)$, 
we need to identify $KK_G(S,B)$ with $KK_G(\C_1,B)$ explicitly, and 
we need the invertible element $\xbf \in KK_G(\C_1,S)$ defined in \cite[Section 19.2]{B}. 
By the extension 
$$\CD
0@>>>S @>>>C@>>>\C@>>> 0,
\endCD$$
we get an invertible element in $KK_G(\C,S\hotimes \C_1)$. 
Then $\xbf$ is the image of this element by the isomorphism 
\begin{eqnarray*}\lefteqn{\tau_{\C_1}:KK_G(\C,S\hotimes \C_1)\rightarrow 
KK_G(\C\hotimes \C_1,S\hotimes\C_1\hotimes \C_1)}\\
&=&KK_G(\C_1,S\hotimes M_2(\C))= KK_G(\C_1,S).
\end{eqnarray*}

For the identification of $\C_1\hotimes \C_1$ and $M_2(\C)$ with standard even grading, we follow the convention 
in the proof of \cite[Theorem 18.10.12]{B} (our computation really depends on it). 
A direct computation shows that $\xbf$ is given by the Kasparov module $(E_2,\phi_2,F_2)\in \E_G(\C_1,S)$ with 
$E_2=\C^2\hotimes (S\oplus S)$, 
$$F_2=\left(
\begin{array}{cc}
0 &1  \\
1 &0 
\end{array}
\right)\otimes \left(
\begin{array}{cc}
1 &0  \\
0 &-1 
\end{array}
\right),$$
$$\phi_2(1)=\left(
\begin{array}{cc}
1 &0  \\
0 &1 
\end{array}
\right)
\otimes Q,\quad 
\phi_2(\varepsilon)
=
\left(
\begin{array}{cc}
0 &-i  \\
i &0 
\end{array}
\right)
\otimes Q,
$$
where the projection $Q\in M_2(M(S))$ is given by 
$$Q(t)=\left(
\begin{array}{cc}
1-t &\sqrt{t(1-t)}  \\
\sqrt{t(1-t)}  &t
\end{array}
\right),
$$
and the grading of $E_2$ is given by 
$$\left(
\begin{array}{cc}
1 &0  \\
0 &-1 
\end{array}
\right)
\otimes 
\left(
\begin{array}{cc}
1 &0  \\
0 &1 
\end{array}
\right).
$$

With this $\xbf$, we have $$\delta''(K_1(j_\beta)([u_{\psi,\varphi}]))
=\delta_{q_n}\hotimes_{\C_1} \xbf\hotimes_S KK_G(\theta)
=\theta_*(\delta_{q_n}\hotimes_{\C_1} \xbf),$$
and so our task now is to compute $\delta_{q_n}\hotimes_{\C_1} \xbf$ explicitly. 

\begin{lemma}\label{quasi-homo} The class $\delta_{q_n}\hotimes_{\C_1} \xbf\in KK_G(\cO_n,S)$ is given by 
the quasi-homomorphism $\rho=(\rho^{(0)},\rho^{(1)})$ from $\cO_n$ to $S$ such that 
$\rho^{(0)}$ and $\rho^{(1)}$ are unital homomorphisms from $\cO_n$ to $\B(H\hotimes S)$ with
$\rho^{(0)}(x)=\Phi(x)\hotimes 1$ and 
$$\rho^{(1)}(x)=\big(P\hotimes 1+(1-P)\hotimes z\big)(\Phi(x)\hotimes 1) \big(P\hotimes 1+(1-P)\hotimes z\big)^*.$$
\end{lemma}

\begin{proof} We regard $H\hotimes S$ as a $\cO_n$-$S$ bimodule with trivial grading, and we set 
$E=(H\hotimes S)\oplus (H\hotimes S)^{op}$. 
We denote by $\Psi:S\rightarrow Q(S\oplus S)$ a Hilbert $S$-module isomorphism given by 
$$\Phi(f)(t)=(\sqrt{1-t}f(t),\sqrt{t}f(t)).$$
Then $E_1\hotimes_{\C_1}E_2$ is identified with $E$ via the identification of 
$(\xi_1\hotimes f_1, \xi_2\hotimes f_2)\in E$ and 
$$\xi_1\hotimes 1\hotimes_{\C_1}(1,0)\hotimes \Psi(f_1)+\xi_2\hotimes 1\hotimes_{\C_1}(0,1)\hotimes \Psi(f_2)
\in H\hotimes \C_1\hotimes_{\C_1} \C^2\hotimes (S\oplus S).$$ 

We claim that $\delta_{q_n}\hotimes_{\C_1} \xbf$ is given by the Kasparov module 
$(E,\phi,F)\in \E_G(\cO_n,S)$ with 
$$\phi(x)=\mathrm{diag}(\Phi(x)\otimes 1,\Phi(x)\otimes 1),$$
$$F=\left(
\begin{array}{cc}
0 &1\hotimes c  \\
1\hotimes c &0 
\end{array}
\right)+
\left(
\begin{array}{cc}
0 &-i(2P-1)\hotimes s  \\
i(2P-1)\hotimes s &0 
\end{array}
\right),
$$
where $c(t)=\cos(\pi t)$, $s(t)=\sin(\pi t)$. 
Indeed, it is easy to show that $(E,\phi,F)$ is a Kasparov module, 
and the graded commutator $[F_1\hotimes 1_{E_2},F]$ is positive. 
We show that $F$ is a $F_2$-connection (see \cite[Definition 18.3.1]{B} for the definition). 
Let $\xi\in H$, $x=(x_1,x_2)\in \C^2$, and $f=(f_1,f_2)\in S\oplus S$. 
Then we have 
$$T_{\xi\hotimes 1}(x\hotimes f)=(x_1\xi\hotimes (\sqrt{1-t}f_1+\sqrt{t}f_2),x_2\xi\hotimes (\sqrt{1-t}f_1+\sqrt{t}f_2))\in E,$$
$$T_{\xi\hotimes \varepsilon}(x\hotimes f)=(-ix_2\xi\hotimes (\sqrt{1-t}f_1+\sqrt{t}f_2),ix_1\xi\hotimes (\sqrt{1-t}f_1+\sqrt{t}f_2))\in E.$$
A direct computation shows that $T_{\xi\hotimes 1}\circ F_2-F\circ T_{\xi\hotimes 1}$ and 
$T_{\xi\hotimes\varepsilon}\circ F_2+F\circ T_{\xi\hotimes \varepsilon}$ are in $\K(E_2,E)$. 
Since $F_2$ and $F$ are self-adjoint, we see that $F$ is a $F_2$-connection. 
Therefore $(E,\phi,F)$ gives the Kasparov product $\delta_{q_n}\hotimes_{\C_1} \xbf$. 

Note that $F$ satisfies $F=F^*$, $F^2=1$. 
Let 
$$U=\left(
\begin{array}{cc}
1\otimes 1 &0  \\
0 &1\hotimes c+i(2P-1)\hotimes s 
\end{array}
\right),
$$which is a unitary in $\B(E)$. 
Then we have 
$$U^*FU=\left(
\begin{array}{cc}
0 &1\otimes 1   \\
1\otimes 1  & 0
\end{array}
\right),
$$
$$U^*\phi(x)U=\left(
\begin{array}{cc}
\rho^{(0)}(x) &0  \\
0 & \rho^{(1)}(x)
\end{array}
\right),
$$
which finish the proof. 
\end{proof}

To continue the proof, we need more detailed information of the homomorphism $\Phi$. 

\begin{lemma} Let the notation be as above. 
\begin{itemize}
\item[(1)] We can choose $\Phi$ so that it has the following form with respect to the orthogonal decomposition 
$H=\cF_n\oplus \cF_n^\perp$: 
$$\Phi(s_i)=\left(
\begin{array}{cc}
t_i &r_i  \\
0 &v_i 
\end{array}
\right).$$
\item[(2)] For $\Phi$ as in (1), the quasi-homomorphism $\rho=(\rho^{(0)},\rho^{(1)})$ in Lemma \ref{quasi-homo} 
is expressed as 
$$\rho^{(0)}(s_i)=\left(
\begin{array}{cc}
t_i\hotimes 1 &r_i\hotimes 1  \\
0 &v_i\hotimes 1 
\end{array}
\right),
\quad 
\rho^{(1)}(s_i)=
\left(
\begin{array}{cc}
t_i\hotimes 1 &r_i\hotimes z^*  \\
0 &v_i\hotimes 1 
\end{array}
\right).
$$
In particular, we have 
$$\sum_{i=1}^n\rho^{(1)}(s_i)\rho^{(0)}(s_i)^*=(1_H-p_n)\hotimes 1+p_n\hotimes z^*. $$
\end{itemize}
\end{lemma}

\begin{proof} (1) 
We first construct $l_n^G:\cO_n\rightarrow \cE_n$ explicitly. 
Ignoring the $G$ actions, we can find a representation $\Phi'$ of $\cO_n$ on $\cF_n\oplus \cF_n$ of the form 
$$\Phi'(s_1)=\left(
\begin{array}{cc}
t_1 &p_n  \\
0 &w_1 
\end{array}
\right),$$
$$\Phi'(s_i)=\left(
\begin{array}{cc}
t_i &0  \\
0 &w_i 
\end{array}
\right),\quad 2\leq i\leq n.$$ 
Using $\Phi'$, we define $l_n$ by 
$$\left(
\begin{array}{cc}
l_n(x) &0  \\
0 &0 
\end{array}
\right)
=\left(
\begin{array}{cc}
1 &0  \\
0 &0 
\end{array}
\right)\Phi'(x)
\left(
\begin{array}{cc}
1 &0  \\
0 &0 
\end{array}
\right),
$$
and $l_n^G$ by $l_n^G(x)=\int_G \tilde{\alpha}_{g^{-1}}\circ l_n\circ \alpha_g(x)dg$. 
We have $l_n^G(s_i)=t_i$ for all $1\leq i\leq n$ by construction.

We show that the Stinespring dilation $(\Phi,H)$ of this $l_n^G$ has the desired property. 
Recall that $H$ is the closure of the algebraic tensor product $\cO_n\odot \cF_n$ with respect to 
the inner product 
$$\inpr{x\odot \xi}{y\odot \eta}=\inpr{l_n^G(y^*x)\xi}{\eta},$$
and $\Phi$ is given by the left multiplication of $\cO_n$. 
The space $\cF_n$ is identified with $1\odot \cF_n$, and the unitary representation $\pi_H$ is given by 
$\pi_H(g)(x\odot \xi)=\alpha_g(x)\odot \pi_{\cF_n}(g)\xi$. 
To show that $\Phi$ has the desired property, it suffices to show $\|s_i\odot \xi -1\odot t_i\xi\|=0$ for 
all $\xi\in \cF_n$. 
Indeed, \begin{eqnarray*}
\lefteqn{\|s_i\odot \xi-1\odot t_i\xi\|^2} \\
 &=&\inpr{l_n^G(s_i^*s_i)\xi}{\xi}-\inpr{l_n^G(s_i)\xi}{t_i\xi}-\inpr{l_n^G(s_i^*)t_i\xi}{\xi}+\inpr{t_i\xi}{t_i\xi}
 =0,
\end{eqnarray*}
and we get the statement. 

(2) The first statement follows from (1) and Lemma \ref{quasi-homo}. 
The Cuntz algebra relation implies 
$$p_nr_i=r_i,\quad r_j^*r_i+v_j^*v_i=\delta_{i,j},$$
$$\sum_{i=1}^nr_ir_i^*=p_n,\quad \sum_{i=1}^n r_iv_i^*=0,\quad \sum_{i=1}^n v_iv_i^*=1. $$
These relations and the first statement imply the second statement. 
\end{proof}

\begin{proof}[Proof of Theorem \ref{boundary}] Thanks to the previous lemma, we may assume that 
the class $\theta_*(\delta_{q_n}\hotimes_{\C_1} \xbf)\in KK_G(\cO_n,B)$ is given by a quasi-homomorphism 
$\sigma=(\sigma^{(0)},\sigma^{(1)})$ from $\cO_n$ to $B$ of the form 
$$\sigma^{(0)}(s_i)=\left(
\begin{array}{cc}
t_i\hotimes 1 &r_i\hotimes 1  \\
0 &v_i\hotimes 1 
\end{array}
\right),
\quad 
\sigma^{(1)}(s_i)=
\left(
\begin{array}{cc}
t_i\hotimes 1 &r_i\hotimes u_{\psi,\varphi}^*  \\
0 &v_i\hotimes 1 
\end{array}
\right),
$$
and they satisfy 
$$\sum_{i=1}^n\sigma^{(1)}(s_i)\sigma^{(0)}(s_i)^*=(1_H-p_n)\hotimes 1+p_n\hotimes u_{\psi,\varphi}^*.$$

We set $\tilde{\sigma}^{(0)}=\sigma^{(0)}\oplus \varphi$, $\tilde{\sigma}^{(1)}=\sigma^{(1)}\oplus \psi$, which 
are unital homomorphisms from $\cO_n$ to  $\B((H\oplus \C)\hotimes B)$. 
Then $\tilde{\sigma}=(\tilde{\sigma}^{(0)},\tilde{\sigma}^{(1)})$ is a quasi-homomorphism with 
$$\sum_{i=1}^n\tilde{\sigma}^{(1)}(s_i)\tilde{\sigma}^{(0)}(s_i)^*=\big((1_H-p_n)\hotimes 1+p_n\hotimes u_{\psi,\varphi}^*\big)
\oplus (1_\C\hotimes u_{\psi,\varphi}),$$
which is denoted by $u$. 
Then we can construct a norm continuous path $\{u_t\}_{t\in [0,1]}$ of unitaries in $\C1+\K(H\oplus \C)^G\otimes B^G$ 
satisfying $u(0)=u$ and $u(1)=1$. 
Let $\tilde{\sigma}_t^{(0)}=\tilde{\sigma}^{(0)}$, and let $\tilde{\sigma}_t^{(1)}$ be the homomorphism from $\cO_n$ to 
$\B((H\oplus \C)\hotimes B)$ determined by $\tilde{\sigma}_t^{(1)}(s_i)=u(t)\tilde{\sigma}_t^{(0)}(s_i)$. 
Then $\tilde{\sigma}_t=(\tilde{\sigma}_t^{(0)},\tilde{\sigma}_t^{(1)})$ gives a homotopy of quasi-homomorphisms connecting 
$\tilde{\sigma}$ and $\tilde{\sigma}_1=(\tilde{\sigma}^{(0)},\tilde{\sigma}^{(0)})$. 
This shows $[\tilde{\sigma}]=0$ in $KK_G(\cO_n,B)$, and so $\theta_*(\delta_{q_n}\hotimes_{\C_1}\xbf)=KK_G(\psi)-KK_G(\varphi)$. 
\end{proof}

\begin{remark} 
The above argument shows that there exists a short exact sequence 
$$0\rightarrow\mathrm{Coker}(1-K_{1-*}(\hat{\beta}_{\pi_\alpha}))\rightarrow KK_G^*(\cO_n,B)\rightarrow 
\mathrm{Ker}(1-K_{*}(\hat{\beta}_{\pi_\alpha}))\rightarrow 0. $$
\end{remark}

\begin{remark}
From (\ref{short}), we obtain the 6-term exact sequence (see \cite[Theorem 4.9]{P}), 
$$\CD
KK_G^0(B,\C)@>1-\hotimes[H_n]>>KK_G^0(B,\C)@>>>KK_G^0(B,\cO_n)\\
@AAA@.@VVV\\
KK_G^1(B,\cO_n)@<<<KK_G^1(B,\C)@<1-\hotimes[H_n]<<KK_G^1(B,\C)
\endCD.
$$
In particular, we have the following exact sequence by setting $B=\C$: 
$$ \CD
0@>>>K_1(\cO_n\rtimes_\alpha G)@>>>\\K^G_0(\C) @>1-\hotimes[H_n]>>K_0^G(\C)
@>>>K_0(\cO_n\rtimes_\alpha G)@>>> 0
\endCD.$$
Let $\iota_\alpha :C^*(G)\rightarrow \cO_n\rtimes_\alpha G$ be the embedding map, 
let $(\pi,H_\pi)$ be an irreducible representation of $G$, and let 
$$e(\pi)_{ij}=\dim \pi \int_G \overline{\pi(g)_{ij}}\lambda_gdg\in C^*(G).$$ 
Then the canonical isomorphism from $K_0^G(\C)$ onto $K_0(C^*(G))$ sends the class of $(\pi,H_\pi)$ 
in $K_0^G(\C)$ to $[e(\overline{\pi})_{11}]\in K_0(C^*(G))$.  
Thus we have the exact sequence 
$$ \CD
0@>>>K_1(\cO_n\rtimes_\alpha G)@>>>\Z\hat{G} @>1-[\overline{\pi_\alpha}]>>\Z\hat{G}
@>>>K_0(\cO_n\rtimes_\alpha G)@>>> 0
\endCD,$$ 
where $[\pi]\in \Z\hat{G}$ is sent to $K_0(\iota_\alpha)([e(\pi)_{11}])\in K_0(\cO_n\rtimes G)$. 
With the identification of $K_*(\cO_n\rtimes_\alpha G)$ and $K_*(\cO_n^G)$, 
this recovers the formula of $K_*(\cO_n^G)$ obtained in \cite{MRS}, \cite{PR}.  
\end{remark}



\begin{thebibliography}{99}

\bibitem{B} Blackadar, B. 
\textit{$K$-theory for operator algebras.} Second edition. Mathematical Sciences Research Institute Publications, 
5. Cambridge University Press, Cambridge, 1998. 

\bibitem{CE} Cuntz, J.; Evans, D. E. 
\textit{Some remarks on the $C^* $-algebras associated with certain topological Markov chains.} 
Math. Scand. \textbf{48} (1981), 235--240.

\bibitem{G} Goldstein, P. 
\textit{Classification of canonical $\Z_2$-actions on $\cO_\infty$.}
preprint, 1997. 

\bibitem{IDuke} Izumi, M. \textit{Finite group actions on $C^*$-algebras 
with the Rohlin property. I.} Duke Math. J. \textbf{122} (2004), 233--280.

\bibitem{IAdv} Izumi, M. \textit{Finite group actions on $C^*$-algebras 
with the Rohlin property. II.} Adv. Math. \textbf{184} (2004), 119--160.

\bibitem{IM} Izumi, M.; Matui, H. 
\textit{$\Z^2$-actions on Kirchberg algebras.} 
Adv. Math. \textbf{224} (2010), 355--400. 

\bibitem{KP} Kirchberg, E.; Phillips, N. C.
\textit{Embedding of exact $C^*$-algebras in the Cuntz algebra $\cO_2$.}
J. Reine Angew. Math. \textbf{525} (2000), 17--53.

\bibitem{K81} Kishimoto, A.
\textit{Outer automorphisms and reduced crossed products of simple
$C^*$-algebras.}
Comm. Math. Phys. \textbf{81} (1981), 429--435.

\bibitem{K98} Kishimoto, A.
\textit{Automorphisms of  AT algebras with the Rohlin property.}
J. Operator Theory \textbf{40} (1998), 277--294.



\bibitem{LP} Lin, Hua Xin; Phillips, N. C. 
\textit{Approximate unitary equivalence of homomorphisms from $\cO_\infty$} 
J. Reine Angew. Math. \textbf{464} (1995), 173--186.

\bibitem{MRS} Mann, M. H.; Raeburn, I.; Sutherland, C. E. 
\textit{Representations of finite groups and Cuntz-Krieger algebras.} 
Bull. Austral. Math. Soc. \textbf{46} (1992), 225--243. 

\bibitem{M} Matui, H. 
\textit{$\Z^N$-actions on UHF algebras of infinite type.} 
to appear in J. Reine Angew. Math. arXiv:1004.3103.

\bibitem{N} Nakamura, H. 
\textit{Aperiodic automorphisms of nuclear purely infinite simple 
$C^*$-algebras}, 
Ergodic Theory Dynam. Systems \textbf{20} (2000), 1749--1765. 

\bibitem{PR} Pask, D.; Raeburn, I. 
\textit{On the $KK$-theory of Cuntz-Krieger algebras.} 
Publ. Res. Inst. Math. Sci. \textbf{32} (1996), 415--443.

\bibitem{Ph} Phillips, N. C.
\textit{A classification theorem for nuclear purely infinite simple
$C^*$-algebras.}
Doc. Math. \textbf{5} (2000), 49--114.

\bibitem{P} Pimsner, M. V. 
\textit{A class of $C^*$-algebras generalizing both Cuntz-Krieger algebras and crossed products by $\Z$}. 
Free probability theory (Waterloo, ON, 1995), 189--212, 
Fields Inst. Commun., 12, Amer. Math. Soc., Providence, RI, 1997. 

\bibitem{R93} R{\o}rdam, M.
\textit{Classification of inductive limits of Cuntz algebras.}
J. Reine Angew. Math. \textbf{440} (1993), 175--200.

\bibitem{R01} R{\o}rdam, M. 
\textit{Classification of Nuclear $C^*$-algebras. 
Entropy in Operator Algebras.} 
Operator Algebras and Non-commutative Geometry VII. 
Encyclopedia of Mathematical Sciences, Springer, 2001.

\end{thebibliography}
\end{document}